\documentclass[12pt,a4paper]{article}
\usepackage[utf8x]{inputenc}
\usepackage{amsmath}
\usepackage{amsfonts}
\usepackage{amssymb}
\usepackage{amsthm}
\usepackage{mathrsfs}
\usepackage{fancyhdr}
\usepackage{graphicx}
\usepackage[usenames,x11names]{xcolor}
\usepackage{arabtex}
\usepackage{framed}
\usepackage[top=2cm,bottom=2cm,margin=2cm]{geometry}
\usepackage{cancel}
\usepackage{array}   

\usepackage[dvipdfm,pdfstartview=FitH,colorlinks=true,linkcolor=Red4,urlcolor=Red4,%
filecolor=green,citecolor=red]{hyperref}

\title{On the derivatives of the integer-valued polynomials}
\author{\sc Bakir FARHI \\
Laboratoire de Mathématiques appliquées \\
Faculté des Sciences Exactes \\
Université de Bejaia, 06000 Bejaia, Algeria \\[1mm]
\href{mailto:bakir.farhi@gmail.com}{bakir.farhi@gmail.com} \\[1mm]
\url{http://farhi.bakir.free.fr/}
}
\date{}

\def\Q{{\mathbb Q}}
\def\C{{\mathbb C}}
\def\N{{\mathbb N}}
\def\Z{{\mathbb Z}}

\def\I{\mathscr{I}}

\def\lcm{\mathrm{lcm}}
\def\id{\mathrm{I}}
\def\den{\mathrm{den}}    

\def\EMdash{\leavevmode\hbox to 10.6mm{\vrule height .63ex depth -.59ex
    width 10mm\hfill}}

\theoremstyle{plain}
\numberwithin{equation}{section}
\newtheorem{thm}{Theorem}[section]

\newtheorem{lemma}[thm]{Lemma}

\newtheorem{rmq}[thm]{Remark}
\newtheorem{prop}[thm]{Proposition}
\newtheorem{coll}[thm]{Corollary}

\begin{document}
\maketitle
\begin{abstract}
In this paper, we study the derivatives of an integer-valued polynomial of a given degree. Denoting by $E_n$ the set of the integer-valued polynomials with degree $\leq n$, we show that the smallest positive integer $c_n$ satisfying the property: $\forall P \in E_n, c_n P' \in E_n$ is $c_n = \lcm(1 , 2 , \dots , n)$. As an application, we deduce an easy proof of the well-known inequality $\lcm(1 , 2 , \dots , n) \geq 2^{n - 1}$ ($\forall n \geq 1$). In the second part of the paper, we generalize our result for the derivative of a given order $k$ and then we give two divisibility properties for the obtained numbers $c_{n , k}$ (generalizing the $c_n$'s). Leaning on this study, we conclude the paper by determining, for a given natural number $n$, the smallest positive integer $\lambda_n$ satisfying the property: $\forall P \in E_n$, $\forall k \in \N$: $\lambda_n P^{(k)} \in E_n$. In particular, we show that: $\lambda_n = \prod_{p \text{ prime}} p^{\lfloor\frac{n}{p}\rfloor}$ ($\forall n \in \N$).
\end{abstract}

\section{Introduction and Notation}
Throughout this paper, we let $\N^*$ denote the set of positive integers. We let $\lfloor\cdot\rfloor$ denote the integer-part function. For a given prime number $p$, we let $v_p$ denote the usual $p$-adic valuation. For a given positive integer $n$ and given positive integers $a_1 , \dots , a_n$, we denote the least common multiple of $a_1 , \dots , a_n$ by $\lcm(a_1 , \dots , a_n)$ or by one of the two equivalent notations $a_1 \vee \dots \vee a_n$ and $\bigvee_{i = 1}^{n} a_i$, which are sometimes more convenient. For given positive integers $n$ and $k$, with $1 \leq k \leq n$, and given real numbers $x_1 , \dots , x_n$, we let $x_1 \cdots \widehat{x_k} \cdots x_n$ denote the product $\displaystyle\prod_{\begin{subarray}{c}
1 \leq i \leq n \\
i \neq k
\end{subarray}} x_i$. We say that a rational number $u$ is a multiple of a non-zero rational number $v$ if the ratio $u / v$ is an integer. For a given rational number $r$, we let $\den(r)$ denote the denominator of $r$; that is the smallest positive integer $d$ such that $d r \in \Z$. For given $n , k \in \N$, with $n \geq k$, we define
$$
F_{n , k} ~:=~ \sum_{\begin{subarray}{c}
i_1 , \dots , i_k \in \N^* \\
i_1 + \dots + i_k = n
\end{subarray}} \frac{1}{i_1 i_2 \cdots i_k} ~~~~\text{and}~~~~ d_{n , k} ~:=~ \den(F_{n , k}) ,
$$
with the conventions that $F_{0 , 0} = 1$ and $F_{n , 0} = 0$ for any $n \in \N^*$. \\
For given $n , k \in \N$, with $n \geq k$, we also define
$$
q_{n , k} ~:=~ \lcm\big\{i_1 i_2 \cdots i_k ~|~ i_1 , \dots , i_k \in \N^* , i_1 + \dots + i_k \leq n\big\} ,
$$
with the convention that $q_{n , 0} = 1$ for any $n \in \N$. Besides, for $n \in \N$, we define
\begin{eqnarray*}
q_n & := & \lcm\big\{q_{n , k} ; 0 \leq k \leq n\big\} \\
& = & \lcm\big\{i_1 i_2 \cdots i_k ~|~ k \in \N^* , i_1 , \dots , i_k \in \N^* , i_1 + \dots + i_k \leq n\big\} .
\end{eqnarray*}

\noindent Next, we let $s(n , k)$ ($n , k \in \N$, $n \geq k$) denote the Stirling numbers of the first kind, which are the integer coefficients appearing in the polynomial identity:
$$
x (x - 1) \cdots (x - n + 1) ~=~ \sum_{k = 0}^{n} s(n , k) x^k
$$
(see e.g., \cite[Chapter V]{com} or \cite[Chapter 6]{gra}).

Further, we let $\id$, $D$ and $\Delta$ the linear operators on $\C[X]$ which respectively represent the identity, the derivation and the forward difference \big($\Delta P(X) = P(X + 1) - P(X) , \forall P \in \C[X]$\big). The expression of $D$ in terms of $\Delta$, obtained by using symbolic methods (see e.g., \linebreak\cite[Chapter 1, §6]{jor}), is given by:
\begin{equation}\label{eq3}
D ~=~ \ln(\id + \Delta) ~=~ \Delta - \frac{\Delta^2}{2} + \frac{\Delta^3}{3} - \dots
\end{equation}
Note that this formula will be of crucial importance throughout this paper.

An integer-valued polynomial is a polynomial $P \in \C[X]$ such that $P(\Z) \subset \Z$; that is the value taken by $P$ at every integer is an integer. It is immediate that every polynomial with integer coefficients is an integer-valued polynomial but the converse is false (for example, the polynomial $\frac{X (X + 1)}{2}$ is a counterexample to the converse statement). However, an integer-valued polynomial always has rational coefficients (i.e., lies in $\Q[X]$). This can be easily proved by using for example the Lagrange interpolation formula. So the set $E$ of the integer-valued polynomials is a subring of $\Q[X]$. More interestingly, the set $E$ can be also seen as a $\Z$-module. From this point of view, it is shown (see e.g., \cite{cah} or \cite{pol}) that $E$ is free with infinite rank and has as a basis the sequence of polynomials:
$$
B_n(X) ~:=~ \dfrac{X (X - 1) \cdots (X - n + 1)}{n!} ~=~ \binom{X}{n}
$$
($n \in \N$), with the convention that $B_0(X) = 1$. \\
From the definition of the polynomials $B_n$ ($n \in \N$), the following identities are immediate:
$$
\Delta^i B_j ~=~ \begin{cases}
B_{j - i} & \text{if } i \leq j \\
0 & \text{else}
\end{cases} ~~~~\text{and}~~~~ B_j(0) ~=~ \begin{cases}
1 & \text{if } j = 0 \\
0 & \text{else}
\end{cases} ~~~~~~~~~~ (\forall i , j \in \N) .
$$
Combining these, we derive that for all $i , j \in \N$, we have
\begin{equation}\label{eq9}
\left(\Delta^i B_j\right)(0) ~=~ \begin{cases}
1 & \text{if } i = j \\
0 & \text{else}
\end{cases} ~=~ \delta_{i j} 
\end{equation}
(where $\delta_{i j}$ denotes the Kronecker delta). The last formula will be useful later in §\ref{sec3}.

For a given $n \in \N$, let $E_n$ denote the set of the integer-valued polynomials with degree $\leq n$. Then, it is clear that $E_n$ is a free submodule of $E$ and has as a basis the polynomials $B_0 , B_1 , \dots , B_n$ (so $E_n$ is of rank $(n + 1)$). Obviously, $E$ is stable by the forward difference operator $\Delta$ (i.e., $\forall P \in E : \Delta P \in E$). The stability by $\Delta$ also holds for each $E_n$ ($n \in \N$). But remarkably, $E$ is not stable by the operator of derivation $D$ (for example $B_2'(X) = X - \frac{1}{2} \not\in E$). This last remark constitutes the starting point of our study. To recover in $E$ (actually in each $E_n$) something that is close to the stability by derivation, we argue like this: for a given $n \in \N$ and a given $P \in E_n$, we can write $P$ as: $P = a_0 B_0 + a_1 B_1 + \dots + a_n B_n$ ($a_0 , \dots , a_n \in \Z$), so we have that $n! P \in \Z[X]$, which implies that $(n! P)' = n! P' \in \Z[X]$. Thus $n! P' \in E_n$. Consequently, the positive integer $n!$ satisfies the following important property:
$$
\forall P \in E_n :~ n! P' \in E_n . 
$$
This leads us to propose the following problem:

\begin{leftbar}
\noindent \textbf{Problem 1:} \emph{%
For a given natural number $n$, determine the smallest positive integer $c_n$ satisfying the property:
$$
\forall P \in E_n :~ c_n P' \in E_n .
$$
}
\end{leftbar}

\noindent In §\ref{sec2}, we show that actually $c_n$ is far enough from $n!$; precisely, we show that: \linebreak $c_n = \lcm(1 , 2 , \dots , n)$. Then, we use this result to derive an easy proof of the nontrivial inequality $\lcm(1 , 2 , \dots , n) \geq 2^{n - 1}$ ($\forall n \geq 1$). In §\ref{sec3}, we first solve the more general problem:

\begin{leftbar}
\noindent \textbf{Problem 2:} \emph{%
For given $n , k \in \N$, determine the smallest positive integer $c_{n , k}$ satisfying the property:
$$
\forall P \in E_n :~ c_{n , k} P^{(k)} \in E_n .
$$
}
\end{leftbar} 

\noindent From the definitions of the $c_n$'s and the $c_{n , k}$'s, it is obvious that $c_{n , 1} = c_n$ ($\forall n \in \N$) and that $c_{n , k} = 1$ if $n , k \in \N$ satisfy the condition $k > n$. The last property allows us to restrict our study of the numbers $c_{n , k}$ to the couples $(n , k) \in \N^2$ such that $n \geq k$. A fundamental result of §\ref{sec3} shows that for every $n , k \in \N$, with $n \geq k$, we have
$$
c_{n , k} ~=~ \lcm\{d_{m , k} ;~ k \leq m \leq n\} .
$$
From this, we deduce that $c_{n , k}$ divides $q_{n , k}$ for any $n , k \in \N$, with $n \geq k$. In the opposite direction, we show that $c_{n , k}$ is a multiple of the rational number $\frac{q_{n , k}}{k!}$ ($\forall n , k \in \N$, $n \geq k$). Then, as a second part of §\ref{sec3}, we solve the following problem:

\begin{leftbar}
\noindent \textbf{Problem 3:} \emph{%
For a given $n \in \N$, determine the smallest positive integer $\lambda_n$ satisfying the property:
$$
\forall P \in E_n , \forall k \in \N :~ \lambda_n P^{(k)} \in E_n .
$$
}
\end{leftbar} 

As a fundamental result, we show that:
$$
\lambda_n ~=~ q_n ~=~ \prod_{p \text{ prime}} p^{\left\lfloor\frac{n}{p}\right\rfloor} 
$$
(for any $n \in \N$).

In §\ref{sec4}, we give some other interesting formulas for the crucial numbers $F_{n , k}$ ($n , k \in \N$, $n \geq k$); in particular, we express the $F_{n , k}$'s in terms of the Stirling numbers of the first kind. We finally conclude the paper by presenting (in tables) the first values of the numbers $c_{n , k}$, $q_{n , k}$ and $\lambda_n$.

\section{Results concerning the first derivative of an integer-valued polynomial}\label{sec2}
In this section, we are going to solve the first problem posed in the introduction. To do so, we need some preparations. For a given $n \in \N$, let:
$$
\I_n ~:=~ \left\{a \in \Z :~ \forall P \in E_n , a P' \in E_n\right\} .
$$
Then, it is easy to check that $\I_n$ is an ideal of $\Z$; besides, $\I_n$ is non-zero because $n! \in \I_n$ (as explained in the introduction). Since $\Z$ is a principal ring, one deduces that $\I_n$ has the form $\I_n = \alpha_n \Z$ ($\alpha_n \in \N^*$), and $\alpha_n$ is simply the smallest positive integer satisfying the property: $\forall P \in E_n , \alpha_n P' \in E_n$. So $\alpha_n$ is nothing else the constant $c_n$ required in Problem 1. Consequently, we have
\begin{equation}\label{eq1}
\I_n ~=~ c_n \Z .
\end{equation}

The following theorem solves Problem 1.
\begin{thm}\label{t1}
For every positive integer $n$, we have
$$
c_n ~=~ \lcm(1 , 2 , \dots , n) .
$$
\end{thm}
\begin{proof}
Let $n \in \N^*$ be fixed. For simplicity, we pose $\ell_n := \lcm(1 , 2 , \dots , n)$. To show that $c_n = \ell_n$, we will show that $\ell_n$ is a multiple of $c_n$ and then that $c_n$ is a multiple of $\ell_n$. \\[1mm]
\textbullet{} Let us show that $\ell_n$ is a multiple of $c_n$; that is $\ell_n \in \I_n$ (in view of \eqref{eq1}). So, according to the definition of $\I_n$, this is equivalent to show the property:
\begin{equation}\label{eq2}
\forall P \in E_n :~ \ell_n P' \in E_n .
\end{equation}
Let us show \eqref{eq2}. So, let $P \in E_n$ and show that $\ell_n P' \in E_n$. By applying the identity of linear operators \eqref{eq3} to $P$, we get
$$
P'~=~ \Delta P - \frac{\Delta^2 P}{2} + \frac{\Delta^3 P}{3} - \dots ~=~ \sum_{k = 1}^{n} \frac{(-1)^{k - 1}}{k} \Delta^k P
$$
(because $\Delta^k P = 0$ for $k > n$). Hence
$$
\ell_n P' ~=~ \sum_{k = 1}^{n} (-1)^{k - 1} \frac{\ell_n}{k} (\Delta^k P) .
$$
Because $\Delta^k P \in E_n$ for any $k \in \N$ (since $E_n$ is stable by $\Delta$) and $\frac{\ell_n}{k} \in \Z$ for any $k \in \{1 , 2 , \dots , n\}$ (by definition of $\ell_n$), the last identity shows that $\ell_n P' \in E_n$, as required. \\[1mm]
\textbullet{} Now, let us show that $c_n$ is a multiple of $\ell_n$. By definition of $\ell_n$, this is equivalent to show that $c_n$ is a multiple of each of the positive integers $1 , 2 , \dots , n$. So, let $k \in \{1 , 2 , \dots , n\}$ be fixed and show that $c_n$ is a multiple of $k$. Since $k \leq n$, we have $B_k \in E_n$; thus (by definition of $c_n$): $c_n B_k' \in E_n$. This implies (in particular) that $c_n B_k'(0) \in \Z$. But since
\begin{multline}
B_k'(0) ~=~ \lim_{x \rightarrow 0} \frac{B_k(x)}{x} ~=~ \lim_{x \rightarrow 0} \frac{(x - 1) (x - 2) \cdots (x - k + 1)}{k!} ~=~ \frac{(-1)(-2) \cdots (- k + 1)}{k!} \\
=~ (-1)^{k - 1} \frac{(k - 1)!}{k!} ~=~ \frac{(-1)^{k - 1}}{k} , \label{eq10}
\end{multline}
it follows that $(-1)^{k - 1} \frac{c_n}{k} \in \Z$, implying that $c_n$ is a multiple of $k$, as required. \\[1mm]
This completes the proof of the theorem.
\end{proof}

As an application of Theorem \ref{t1}, we derive a well-known nontrivial lower bound for $\lcm(1 , 2 , \dots , n)$ ($n \in \N^*$). We have the following:
\begin{coll}\label{coll1}
For every positive integer $n$, we have
$$
\lcm(1 , 2 , \dots , n) ~\geq~ 2^{n - 1} .
$$
\end{coll}
To deduce this corollary from Theorem \ref{t1}, we just need the special identity of the following lemma.

\begin{lemma}\label{l1}
For every positive integer $n$, we have
$$
\frac{1}{n} \sum_{k = 0}^{n - 1} \frac{1}{|B_n'(k)|} ~=~ 2^{n - 1} .
$$
\end{lemma}
\begin{proof}
Let $n \in \N^*$ be fixed. From the definition $B_n(X) := \frac{X (X - 1) \cdots (X - n + 1)}{n!}$, we derive that:
$$
B_n'(X) ~=~ \sum_{\ell = 0}^{n - 1} \frac{X \cdots \widehat{(X - \ell)} \cdots (X - n + 1)}{n!} .
$$
It follows that for any $k \in \{0 , 1 , \dots , n - 1\}$, we have
\begin{eqnarray*}
B_n'(k) & = & {\left[\frac{X \cdots \widehat{(X - k)} \cdots (X - n + 1)}{n!}\right]}_{X = k} \\[3mm]
& = & \frac{k (k - 1) \cdots 1 \times (-1) (-2) \cdots (k - n + 1)}{n!} \\[3mm]
& = & (-1)^{n - k - 1} \, \frac{k! (n - k - 1)!}{n!} ,
\end{eqnarray*}
which gives
$$
\frac{1}{|B_n'(k)|} ~=~ \frac{n!}{k! (n - k - 1)!} ~=~ n \binom{n - 1}{k} .
$$
Thus:
$$
\frac{1}{n} \sum_{k = 0}^{n - 1} \frac{1}{|B_n'(k)|} ~=~ \sum_{k = 0}^{n - 1} \binom{n - 1}{k} ~=~ 2^{n - 1}
$$
(according to the binomial formula). The lemma is proved.
\end{proof}

\begin{proof}[Proof of Corollary \ref{coll1}]
Let $n \in \N^*$ be fixed. Since $B_n \in E_n$, we have $c_n B_n' \in E_n$; that is $c_n B_n'(k) \in \Z$ for any $k \in \Z$. In particular, we have $c_n B_n'(k) \in \Z$ for any $k \in \{0 , 1 , \dots , n - 1\}$. But since $B_n'(k) \neq 0$ for $k \in \{0 , 1 , \dots , n - 1\}$ (see the proof of the preceding lemma), we have precisely $c_n B_n'(k) \in \Z^*$ ($\forall k \in \{0 , 1 , \dots , n - 1\}$), implying that $|c_n B_n'(k)| \geq 1$ ($\forall k \in \{0 , 1 , \dots , n - 1\}$). Using this fact, we get
$$
\frac{1}{n} \sum_{k = 0}^{n - 1} \frac{1}{|c_n B_n'(k)|} ~\leq~ \frac{1}{n} \sum_{k = 0}^{n - 1} 1 ~=~ 1 .
$$
But according to Lemma \ref{l1}, we have
$$
\frac{1}{n} \sum_{k = 0}^{n - 1} \frac{1}{|c_n B_n'(k)|} ~=~ \frac{2^{n - 1}}{c_n} .
$$
Thus $\frac{2^{n - 1}}{c_n} \leq 1$, which gives $c_n \geq 2^{n - 1}$; that is (according to Theorem \ref{t1}) $\lcm(1 , 2 , \dots , n) \geq 2^{n - 1}$, as required. The corollary is proved.
\end{proof}

\section{Results concerning the higher order derivatives of an integer-valued polynomial}\label{sec3}

In this section, we are going to solve the second problem posed in the introduction. To do so, we just adapt and generalize the method used in §\ref{sec2}. For given $n , k \in \N$, let
$$
\I_{n , k} ~:=~ \left\{a \in \Z :~ \forall P \in E_n , a P^{(k)} \in E_n\right\} .
$$
It is easy to check that $\I_{n , k}$ is an ideal of $\Z$. Besides, for any $P \in E_n$, we have $n! P \in \Z[X]$ (as explained in the introduction), which implies that $(n! P)^{(k)} = n! P^{(k)} \in \Z[X]$ and so $n! P^{(k)} \in E_n$. Hence $n! \in \I_{n , k}$, showing that the ideal $\I_{n , k}$ is non-zero. Since $\Z$ is a principal ring, one deduces that $\I_{n , k}$ has the form $\I_{n , k} = \alpha_{n , k} \Z$ ($\alpha_{n , k} \in \N^*$) and $\alpha_{n , k}$ is simply the smallest positive integer satisfying the property: $\forall P \in E_n$, $\alpha_{n , k} P^{(k)} \in E_n$. So $\alpha_{n , k}$ is nothing else the constant $c_{n , k}$ required in Problem 2. Thus, we have
\begin{equation}\label{eq4}
\I_{n , k} ~=~ c_{n , k} \Z .
\end{equation}
The following theorem solves Problem 2.
\begin{thm}\label{t2}
For every natural numbers $n$ and $k$, we have
$$
c_{n , k} ~=~ \lcm\left\{d_{m , k} ~;~ k \leq m \leq n\right\} .
$$
In particular, $c_{n , k}$ divides the positive integer $q_{n , k}$.
\end{thm}
\begin{proof}
Let $n , k \in \N$ be fixed. For simplicity, we pose $\ell_{n , k} := \lcm\{d_{m , k} ;~ k \leq m \leq n\}$. To show that $c_{n , k} = \ell_{n , k}$, we will show that $\ell_{n , k}$ is a multiple of $c_{n , k}$ and then that $c_{n , k}$ is a multiple of $\ell_{n , k}$. \\[1mm]
\textbullet{} Let us show that $\ell_{n , k}$ is a multiple of $c_{n , k}$; that is $\ell_{n , k} \in \I_{n , k}$ (in view of \eqref{eq4}). So, according to the definition of $\I_{n , k}$, this is equivalent to show the property:
\begin{equation}\label{eq5}
\forall P \in E_n :~ \ell_{n , k} P^{(k)} \in E_n .
\end{equation}
Let us show \eqref{eq5}. So, let $P \in E_n$ and show that $\ell_{n , k} P^{(k)} \in E_n$. From \eqref{eq3}, we derive the following identity of linear operators on $\C[X]$:
\begin{eqnarray*}
D^k & = & \left(\sum_{i \in \N^*} \frac{(-1)^{i - 1}}{i} \Delta^i\right)^k \\
& = & \left(\sum_{i_1 \in \N^*} \frac{(-1)^{i_1 - 1}}{i_1} \Delta^{i_1}\right) \left(\sum_{i_2 \in \N^*} \frac{(-1)^{i_2 - 1}}{i_2} \Delta^{i_2}\right) \cdots \left(\sum_{i_k \in \N^*} \frac{(-1)^{i_k - 1}}{i_k} \Delta^{i_k}\right) \\
& = & \sum_{i_1 , \dots , i_k \in \N^*} \frac{(-1)^{i_1 + \dots + i_k - k}}{i_1 i_2 \cdots i_k} \Delta^{i_1 + \dots + i_k} .
\end{eqnarray*}
Applying this to $P$, we obtain (since $\Delta^i P = 0$ for $i > n$) that:
\begin{eqnarray}
P^{(k)} & = & \sum_{\begin{subarray}{c}
i_1 , \dots , i_k \in \N^* \\
i_1 + \dots + i_k \leq n
\end{subarray}} \frac{(-1)^{i_1 + \dots + i_k - k}}{i_1 i_2 \cdots i_k} \Delta^{i_1 + \dots + i_k} P \notag \\
& = & \sum_{k \leq m \leq n} (-1)^{m - k} \left(\sum_{\begin{subarray}{c}
i_1 , \dots , i_k \in \N^* \\
i_1 + \dots + i_k = m
\end{subarray}} \frac{1}{i_1 i_2 \cdots i_k}\right) \Delta^m P \notag \\
& = & \sum_{k \leq m \leq n} (-1)^{m - k} F_{m , k} \, \Delta^m P . \label{eq6} 
\end{eqnarray}
Because $\Delta^m P \in E_n$ for any $m \in \N$ (since $E_n$ is stable by $\Delta$) and $\ell_{n , k} F_{m , k} \in \Z$ for any $m \in \{k , k + 1 , \dots , n\}$ (according to the definition of $\ell_{n , k}$), the last identity shows that $\ell_{n , k} P^{(k)} \in E_n$, as required. \\[1mm]
\textbullet{} Now, let us show that $c_{n , k}$ is a multiple of $\ell_{n , k}$. By definition of $\ell_{n , k}$, this is equivalent to show that $c_{n , k}$ is a multiple of each of the positive integers $d_{m , k}$ ($k \leq m \leq n$). So, let $m_0 \in \{k , \dots , n\}$ be fixed and show that $c_{n , k}$ is a multiple of $d_{m_0 , k}$. Since $m_0 \leq n$, we have $B_{m_0} \in E_n$; thus (by definition of $c_{n , k}$): $c_{n , k} B_{m_0}^{(k)} \in E_n$. This implies (in particular) that $c_{n , k} B_{m_0}^{(k)}(0) \in \Z$. But, by applying \eqref{eq6} for $P = B_{m_0}$ and using \eqref{eq9}, we get
\begin{eqnarray*}
B_{m_0}^{(k)}(0) & = & \sum_{1 \leq m \leq n} (-1)^{m - k} F_{m , k} \left(\Delta^m B_{m_0}\right)(0) \\
& = & (-1)^{m_0 - k} F_{m_0 , k} .
\end{eqnarray*}
Thus $c_{n , k} \cdot (-1)^{m_0 - k} F_{m_0 , k} \in \Z$, implying that $c_{n , k}$ is a multiple of $\den(F_{m_0 , k}) = d_{m_0 , k}$, as required. So, the first part of the theorem is proved. \\
Next, the second part of the theorem immediately follows from its first part and the trivial fact that $d_{m , k}$ divides $\lcm\{i_1 i_2 \cdots i_k |~ i_1 , \dots , i_k \in \N^* , i_1 + \dots + i_k = m\}$, which divides $q_{n , k}$ (for every $m , k \in \N$, with $k \leq m \leq n$). This achieves the proof of the theorem.
\end{proof}

Concerning the divisibility relations between the $c_{n , k}$'s and the $q_{n , k}$'s, we also have the following result:
\begin{thm}\label{t3}
For every natural numbers $n$ and $k$ such that $n \geq k$, the positive integer $c_{n , k}$ is a multiple of the rational number $\frac{q_{n , k}}{k!}$.
\end{thm}
\begin{proof}
Let $n , k \in \N$ be fixed such that $n \geq k$. Show that the positive integer $c_{n , k}$ is a multiple of the rational number $\frac{q_{n , k}}{k!}$ is equivalent to show that the positive integer $k! c_{n , k}$ is a multiple of the positive integer $q_{n , k}$, which is equivalent (according to the definition of $q_{n , k}$) to show that $k! c_{n , k}$ is a multiple of each of the positive integers having the form $i_1 i_2 \cdots i_k$, where $i_1 , \dots , i_k \in \N^*$ and $i_1 + \dots + i_k \leq n$. So, let $i_1 , \dots , i_k \in \N^*$ such that $i_1 + \dots + i_k \leq n$ and show that $k! c_{n , k}$ is a multiple of the product $i_1 i_2 \cdots i_k$. To do so, let us consider the integer-valued polynomial
$$
P(X) ~:=~ \binom{X}{i_1} \binom{X}{i_2} \cdots \binom{X}{i_k} ~=~ B_{i_1}(X) B_{i_2}(X) \cdots B_{i_k}(X)
$$
whose degree is $i_1 + \dots + i_k \leq n$, showing that $P \in E_n$. \\
Since the expansion of each polynomial $B_i$ ($i \in \N^*$) in the canonical basis $(1 , X , X^2 , \dots)$ of $\Q[X]$ begins with
$$
\frac{(-1)^{i - 1}}{i} X + \dots
$$
(because $B_i(0) = 0$ and $B_i'(0) = \frac{(-1)^{i - 1}}{i}$, according to \eqref{eq10}) then the expansion of the polynomial $P$ in the canonical basis of $\Q[X]$ begins with
$$
\frac{(-1)^{i_1 - 1}}{i_1} \cdot \frac{(-1)^{i_2 - 1}}{i_2} \cdots \frac{(-1)^{i_k - 1}}{i_k} X^k + \dots ~=~ \pm \frac{1}{i_1 i_2 \cdots i_k} X^k + \dots
$$
It follows from this fact that we have
$$
P^{(k)}(0) ~=~ \pm \frac{k!}{i_1 i_2 \cdots i_k} .
$$
On the other hand, since $c_{n , k} P^{(k)} \in E_n$, we have $c_{n , k} P^{(k)}(0) \in \Z$; that is
$$
\pm c_{n , k} \frac{k!}{i_1 i_2 \cdots i_k} \in \Z ,
$$
showing that $k! c_{n , k}$ is a multiple of $i_1 i_2 \cdots i_k$, as required. This completes the proof of the theorem.
\end{proof}

\begin{rmq}
Theorem \ref{t1} can immediately follow from Theorems \ref{t2} and \ref{t3}. Indeed, by applying the second part of Theorem \ref{t2} and Theorem \ref{t3} for $k = 1$, we obtain that for any $n \in \N$, the positive integer $c_{n , 1} = c_n$ is both a divisor and a multiple of the positive integer $q_{n , 1} = \lcm(1 , 2 , \dots , n)$. So, we have $c_n = \lcm(1 , 2 , \dots , n)$ ($\forall n \in \N$), which is nothing else the result of Theorem \ref{t1}.
\end{rmq}

Now, we are going to solve Problem 3 and prove the result announced in the introduction. We have the following:
\begin{thm}\label{t4}
For every natural number $n$, we have
$$
\lambda_n ~=~ q_n ~=~ \prod_{p \text{ prime}} p^{\left\lfloor\frac{n}{p}\right\rfloor} .
$$
\end{thm} 

The proof of Theorem \ref{t4} needs the two following lemmas.

\begin{lemma}\label{l2}
For every positive integer $a$ and every prime number $p$, we have
$$
v_p(a) ~\leq~ \frac{a}{p} .
$$
\end{lemma}
\begin{proof}
Let $a$ be a positive integer and $p$ be a prime number. Setting $\alpha := v_p(a)$, we can write $a = p^{\alpha} b$, where $b$ is a positive integer which is not a multiple of $p$. For $\alpha = 0$, the inequality of the lemma is trivial. Next, for $\alpha \geq 1$, we have
$$
v_p(a) ~=~ \alpha ~\leq~ 2^{\alpha - 1} ~\leq~ p^{\alpha - 1} ~\leq~ p^{\alpha - 1} b ~=~ \frac{a}{p} ,
$$
as required. This completes the proof of the lemma.
\end{proof}

\begin{lemma}[The key lemma]\label{l3}
For any positive integer $k$ and any prime number $p$, we have
$$
v_p\left(F_{k p , k}\right) ~=~ - k .
$$
\end{lemma}
\begin{proof}
Let $k$ be a positive integer and $p$ be a prime number. We have by definition:
$$
F_{k p , k} ~:=~ \sum_{\begin{subarray}{c}
i_1 , \dots , i_k \in \N^* \\
i_1 + \dots + i_k = k p
\end{subarray}} \frac{1}{i_1 i_2 \cdots i_k} .
$$
For any given $k$-uplet $(i_1 , \dots , i_k) \in \N^{* k}$ such that $i_1 + \dots + i_k = k p$, we have
\begin{eqnarray*}
v_p\left(\frac{1}{i_1 i_2 \cdots i_k}\right) & = & - \sum_{r = 1}^{k} v_p(i_r) \\
& \geq & - \sum_{r = 1}^{k} p^{v_p(i_r) - 1} \\
& \geq & - \sum_{r = 1}^{k} \frac{i_r}{p} \\
& = & - \frac{k p}{p} ~=~ - k .
\end{eqnarray*}
Besides, the last series of inequalities shows that the equality $v_p(\frac{1}{i_1 i_2 \cdots i_k}) = - k$ holds if and only if we have for any $r \in \{1 , \dots , k\}$:
$$
v_p(i_r) ~=~ p^{v_p(i_r) - 1} ~~\text{and}~~ p^{v_p(i_r)} ~=~ i_r .
$$
This condition is clearly satisfied if $(i_1 , \dots , i_k) = (p , \dots , p)$. Conversely, if the condition in question is satisfied then each $i_r$ ($r = 1 , \dots , k$) is a power of $p$ and not equal to $1$. This implies in particular that $i_r \geq p$ ($\forall r \in \{1 , \dots , k\}$). But since $i_1 + \dots + i_k = k p$, we necessarily have $i_1 = i_2 = \dots = i_k = p$. Consequently, the equality $v_p(\frac{1}{i_1 i_2 \cdots i_k}) = - k$ holds if and only if $(i_1 , \dots , i_k) = (p , \dots , p)$. It follows (according to the elementary properties of the usual $p$-adic valuation) that:
$$
v_p(F_{k p , k}) ~=~ \min_{\begin{subarray}{c}
i_1 , \dots , i_k \in \N^* \\
i_1 + \dots + i_k = k p
\end{subarray}} v_p\left(\frac{1}{i_1 i_2 \cdots i_k}\right) ~=~ - k ,
$$
as required. The lemma is proved.
\end{proof}

\begin{proof}[Proof of Theorem \ref{t4}]
Let $n$ be a fixed natural number. From the definition of $\lambda_n$, it is clear that $\lambda_n$ is the smallest positive integer belonging to the ideal of $\Z$:
\begin{eqnarray*}
\bigcap_{k \in \N} \I_{n , k} & = & \bigcap_{k \in \N} c_{n , k} \Z ~~~~~~~~~~~~~~~~~~~~~~~~~~~~~~~~ \text{(according to \eqref{eq4})} \\
& = & \lcm\big\{c_{n , k} ~;~ k \in \N\big\} \Z \\
& = & \lcm\big\{c_{n , k} ~;~ 0 \leq k \leq n\big\} \Z ~~~~~~~~ \text{(since $c_{n , k} = 1$ for $k > n$)} .
\end{eqnarray*}
Thus, we have
\begin{equation}\label{eq11}
\lambda_n ~=~ \lcm\big\{c_{n , k} ~;~ 0 \leq k \leq n\big\} .
\end{equation}
Using Theorem \ref{t2}, we derive that:
\begin{equation}\label{eq12}
\lambda_n ~=~ \lcm\big\{d_{m , k} ~;~ 0 \leq k \leq m \leq n\big\} .
\end{equation}
Now, from \eqref{eq11} and the second part of Theorem \ref{t2}, we immediately derive that $\lambda_n$ divides the positive integer $\lcm\big\{q_{n , k} ~;~ 0 \leq k \leq n\big\} = q_n$. So, to complete the proof of Theorem \ref{t4}, it remains to prove that $\lambda_n$ is a multiple of $q_n$ and that $q_n = \prod_{p \text{ prime}} p^{\lfloor\frac{n}{p}\rfloor}$. This is clearly equivalent to prove that for any prime number $p$, we have
\begin{align}
v_p(q_n) & \leq~ \left\lfloor\frac{n}{p}\right\rfloor , \tag{$I$} \\
v_p(q_n) & \geq~ \left\lfloor\frac{n}{p}\right\rfloor , \tag{$II$} \\
v_p(\lambda_n) & \geq~ \left\lfloor\frac{n}{p}\right\rfloor . \tag{$III$} 
\end{align}
Let $p$ be a prime number and let us begin with proving $(I)$. Since
$$
q_n ~=~ \lcm\big\{i_1 i_2 \cdots i_k ~|~ k \in \N^* , i_1 , \dots , i_k \in \N^* , i_1 +\dots + i_k \leq n\big\} ,
$$
then we have
\begin{equation}\label{eq13}
v_p(q_n) ~=~ \max_{\begin{subarray}{c}
k \in \N^* , i_1 , \dots , i_k \in \N^* \\
i_1 + \dots + i_k \leq n
\end{subarray}} v_p(i_1 i_2 \cdots i_k) .
\end{equation}
Next, for any $k \in \N^*$ and any $i_1 , \dots , i_k \in \N^*$ such that $i_1 + \dots + i_k \leq n$, we have
\begin{eqnarray*}
v_p(i_1 i_2 \cdots i_k) & = & v_p(i_1) + v_p(i_2) + \dots + v_p(i_k) \\
& \leq & \frac{i_1}{p} + \frac{i_2}{p} + \dots + \frac{i_k}{p} ~~~~~~~~ \text{(according to Lemma \ref{l2})} \\
& = & \frac{i_1 + \dots + i_k}{p} ~\leq~ \frac{n}{p} ;
\end{eqnarray*}
that is (since $v_p(i_1 i_2 \cdots i_k) \in \N$):
$$
v_p\left(i_1 i_2 \cdots i_k\right) ~\leq~ \left\lfloor\frac{n}{p}\right\rfloor . 
$$
Hence (according to \eqref{eq13}):
$$
v_p(q_n) ~\leq~ \left\lfloor\frac{n}{p}\right\rfloor ,
$$
as required by $(I)$. \\
Now, let us prove $(II)$. For $p > n$, the inequality $(II)$ is trivial; so suppose that $p \leq n$ and let $\ell := \lfloor\frac{n}{p}\rfloor \geq 1$ and $i_1 = i_2 = \dots = i_{\ell} = p$. Since $\ell \in \N^*$, $i_1 , \dots , i_{\ell} \in \N^*$ and $i_1 + \dots + i_{\ell} = \ell p \leq n$, then $q_n$ is a multiple of the product $i_1 i_2 \cdots i_{\ell} = p^{\ell}$. Thus
$$
v_p(q_n) ~\geq~ v_p(p^{\ell}) = \ell = \left\lfloor\frac{n}{p}\right\rfloor ,
$$
as required by $(II)$. \\
Let us finally prove $(III)$. According to \eqref{eq12}, we have
\begin{eqnarray*}
v_p(\lambda_n) & = & \max_{0 \leq k \leq m \leq n} v_p\left(d_{m , k}\right) \\
& \geq & v_p\left(d_{p \lfloor\frac{n}{p}\rfloor , \lfloor\frac{n}{p}\rfloor}\right) \\
& \geq & - v_p\left(F_{p \lfloor\frac{n}{p}\rfloor , \lfloor\frac{n}{p}\rfloor}\right)
\end{eqnarray*}
(since $d_{p \lfloor\frac{n}{p}\rfloor , \lfloor\frac{n}{p}\rfloor}$ is the denominator of the rational number $F_{p \lfloor\frac{n}{p}\rfloor , \lfloor\frac{n}{p}\rfloor}$). But, from Lemma \ref{l3}, we have that:
$$
v_p\left(F_{p \lfloor\frac{n}{p}\rfloor , \lfloor\frac{n}{p}\rfloor}\right) ~=~ - \left\lfloor\frac{n}{p}\right\rfloor .
$$
Hence
$$
v_p(\lambda_n) ~\geq~ \left\lfloor\frac{n}{p}\right\rfloor ,
$$
as required by $(III)$. \\
This completes the proof of Theorem \ref{t4}.
\end{proof}

\section{Some other formulas for the numbers $F_{n , k}$ and tables of the $c_{n , k}$'s, the $q_{n , k}$'s and the $\lambda_n$'s}\label{sec4}

The following proposition gives some other useful formulas for the numbers $F_{n , k}$ ($n , k \in \N$, $n \geq k$):

\begin{prop}\label{p1}
For every $n , k \in \N$, with $n \geq k$, we have
\begin{eqnarray}
F_{n , k} & = & (-1)^{n + k} \frac{k!}{n!} s(n , k) ~=~ \frac{k!}{n!} \left\vert s(n , k)\right\vert , \label{eq14} \\[2mm]
F_{n , k} & = & \left\vert\binom{X}{n}^{(k)}(0)\right\vert . \label{eq15}
\end{eqnarray}
If in addition $k \geq 2$, then we have
\begin{equation}\label{eq16}
F_{n , k} ~=~ \frac{k!}{n} \sum_{1 \leq i_1 < i_2 < \dots < i_{k - 1} \leq n - 1} \frac{1}{i_1 i_2 \cdots i_{k - 1}} .
\end{equation}
\end{prop}

\begin{proof}
Let us prove Formula \eqref{eq14}. To do so, consider the following formal power series generating function in two indeterminates:
$$
S(X , Y) ~:=~ \sum_{\begin{subarray}{c}
n , k \in \N \\
n \geq k
\end{subarray}} F_{n , k} X^n \frac{Y^k}{k!} .
$$
Using the definition of the $F_{n , k}$'s, we have
\begin{eqnarray*}
S(X , Y) & = & \sum_{\begin{subarray}{c}
n , k \in \N \\
n \geq k
\end{subarray}} \left(\sum_{\begin{subarray}{c}
i_1 , \dots , i_k \in \N^* \\
i_1 + \dots + i_k = n
\end{subarray}} \frac{1}{i_1 i_2 \cdots i_k}\right) X^n \frac{Y^k}{k!} \\
& = & \sum_{k \in \N} \left(\sum_{i_1 , \dots , i_k \in \N^*} \frac{X^{i_1 + \dots + i_k}}{i_1 i_2 \cdots i_k}\right) \frac{Y^k}{k!} \\
& = & \sum_{k \in \N} \left(\sum_{i = 1}^{+ \infty} \frac{X^i}{i}\right)^k \frac{Y^k}{k!} \\
& = & \sum_{k \in \N} \left(- \log(1 - X)\right)^k \frac{Y^k}{k!}
\end{eqnarray*}
\begin{eqnarray*}
& = & \sum_{k \in \N} \frac{\left(- Y \log(1 - X)\right)^k}{k!} \\
& = & e^{- Y \log(1 - X)} \\
& = & (1 - X)^{- Y} \\
& = & \sum_{n = 0}^{+ \infty} \binom{- Y}{n} (- X)^n ~~~~~~~~ \text{(according to the generalized binomial theorem)} \\
& = & \sum_{n = 0}^{+ \infty} \frac{1}{n!} \left(\sum_{k = 0}^{n} s(n , k) (- Y)^k\right) (- X)^n \\
& = & \sum_{\begin{subarray}{c}
n , k \in \N \\
n \geq k
\end{subarray}} \frac{(-1)^{n + k}}{n!} s(n , k) X^n Y^k .
\end{eqnarray*}
By identifying the coefficients in the first and last formal power series of the above series of equalities, we derive that for all $n , k \in \N$, with $n \geq k$, we have
$$
F_{n , k} ~=~ (-1)^{n + k} \frac{k!}{n!} s(n , k) ,
$$
as required by \eqref{eq14}. The second equality of \eqref{eq14} simply follows from the positivity of the $F_{n , k}$'s. \\
Now, let us prove Formula \eqref{eq15}. For any given $n , k \in \N$, with $n \geq k$, we have
$$
\binom{X}{n} ~=~ \frac{1}{n!} X (X - 1) \cdots (X - n + 1) ~=~ \frac{1}{n!} \sum_{i = 0}^{n} s(n , i) X^i .
$$
Thus
$$
\binom{X}{n}^{(k)} ~=~ \frac{1}{n!} \sum_{k \leq i \leq n} s(n , i) i (i - 1) \cdots (i - k + 1) X^{i - k} ,
$$
which gives
$$
\binom{X}{n}^{(k)}(0) ~=~ \frac{k!}{n!} s(n , k)
$$
and concludes (according to \eqref{eq14}) that
$$
\left\vert\binom{X}{n}^{(k)}(0)\right\vert ~=~ \frac{k!}{n!} |s(n , k)| ~=~ F_{n , k} ,
$$
as required by \eqref{eq15}. \\
Let us finally prove \eqref{eq16}. For any given $n , k \in \N$, with $n \geq k \geq 2$, the sum
$$
\sum_{1 \leq i_1 < i_2 < \dots < i_{k - 1} \leq n - 1} \frac{1}{i_1 i_2 \cdots i_{k - 1}}
$$
is nothing else the coefficient of $X^{n - k}$ in the polynomial $(X + \frac{1}{1}) (X + \frac{1}{2}) \cdots (X + \frac{1}{n - 1})$ of $\Q[X]$, which is also the coefficient of $\frac{1}{X^{n - k}} = X^{k - n}$ in the rational fraction $(\frac{1}{X} + \frac{1}{1}) (\frac{1}{X} + \frac{1}{2}) \cdots (\frac{1}{X} + \frac{1}{n - 1})$ of $\Q[\frac{1}{X}]$, when expanded in the basis $(1 , \frac{1}{X} , \frac{1}{X^2} , \dots)$. But since we have
\begin{eqnarray*}
\left(\frac{1}{X} + \frac{1}{1}\right) \left(\frac{1}{X} + \frac{1}{2}\right) \cdots \left(\frac{1}{X} + \frac{1}{n - 1}\right) & = & \frac{1}{(n - 1)! X^n} X (X + 1) \cdots (X + n - 1) \\
& = & \frac{1}{(n - 1)! X^n} \sum_{i = 0}^{n} |s(n , i)| X^i \\
& = & \sum_{i = 0}^{n} \frac{|s(n , i)|}{(n - 1)!} X^{i - n} ,
\end{eqnarray*}
then the coefficient of $X^{k - n}$ in this expression is also equal to $\frac{|s(n , k)|}{(n - 1)!}$. Thus, we have
$$
\frac{|s(n , k)|}{(n - 1)!} ~=~ \sum_{1 \leq i_1 < i_2 < \dots < i_{k - 1} \leq n - 1} \frac{1}{i_1 i_2 \cdots i_{k - 1}} .
$$
Then, Formula \eqref{eq16} immediately follows by using the second equality of Formula \eqref{eq14}. \\
The proof of the proposition is complete.
\end{proof}

\subsection*{Tables}
Now, we explain how we can easily calculate the numbers $F_{n , k}$, $d_{n , k}$, $c_{n , k}$, $q_{n , k}$ and $\lambda_n$ ($n , k \in \N$, $n \geq k$). Starting from the well-known recurrent relation for Stirling numbers of the first kind:
$$
s(n + 1 , k) ~=~ s(n , k - 1) - n \, s(n , k) ~~~~~~~~~~ (\forall n , k \in \N^*, \text{ with } n \geq k)
$$
and using Formula \eqref{eq14} of Proposition \ref{p1}, we immediately show that the rational numbers $F_{n , k}$ satisfy the recurrent relation:
\begin{equation}\label{eq17}
F_{n + 1 , k} ~=~ \frac{k}{n + 1} F_{n , k - 1} + \frac{n}{n + 1} F_{n , k} ~~~~~~~~~~ (\forall n , k \in \N^*, \text{ with } n \geq k) . 
\end{equation}
This last relation allows us to easily generate, via some programming language, the numbers $F_{n , k}$. Then, to generate the positive integers $d_{n , k}$, we simply use their definition: \linebreak $d_{n , k} := \den(F_{n , k})$ and to generate the positive integers $c_{n , k}$, we use the Formula of Theorem \ref{t2}: $c_{n , k} = \lcm\{d_{m , k} ;~ k \leq m \leq n\}$. Because $c_{n , k} = 0$ for $n < k$, it is practical to arrange the $c_{n , k}$'s (for $0 \leq k \leq n$) in a triangular array in which each $c_{n , k}$ is the entry in the $n$\textsuperscript{th} row and $k$\textsuperscript{th} column. The calculations (using Maple software) give the following triangle of the $c_{n , k}$'s ($n \geq k$) up to $n = 10$:\pagebreak

\begin{table}[t]
\begin{tabular}{p{12mm}p{12mm}p{12mm}p{12mm}p{12mm}p{12mm}p{12mm}p{12mm}p{12mm}p{12mm}p{12mm}}
$1$ & ~ & ~ & ~ & ~ & ~ & ~ & ~ & ~ & ~ & ~ \\
$1$ & $1$ & ~ & ~ & ~ & ~ & ~ & ~ & ~ & ~ & ~ \\
$1$ & $2$ & $1$ & ~ & ~ & ~ & ~ & ~ & ~ & ~ & ~ \\
$1$ & $6$ & $1$ & $1$ & ~ & ~ & ~ & ~ & ~ & ~ & ~ \\
$1$ & $12$ & $12$ & $2$ & $1$ & ~ & ~ & ~ & ~ & ~ & ~ \\
$1$ & $60$ & $12$ & $4$ & $1$ & $1$ & ~ & ~ & ~ & ~ & ~ \\
$1$ & $60$ & $180$ & $8$ & $6$ & $2$ & $1$ & ~ & ~ & ~ & ~ \\
$1$ & $420$ & $180$ & $120$ & $6$ & $6$ & $1$ & $1$ & ~ & ~ & ~ \\
$1$ & $840$ & $5040$ & $240$ & $240$ & $6$ & $4$ & $2$ & $1$ & ~ & ~ \\
$1$ & $2520$ & $5040$ & $15120$ & $240$ & $144$ & $4$ & $12$ & $1$ & $1$ & ~ \\
$1$ & $2520$ & $25200$ & $30240$ & $15120$ & $288$ & $240$ & $24$ & $3$ & $2$ & $1$
\end{tabular} \\
\caption{The triangle of the $c_{n , k}$'s for $0 \leq k \leq n \leq 10$}
\end{table}

\begin{rmq}\label{r1}
Another way to generate the $F_{n , k}$'s consists to use the following recurrent relation:
$$
F_{n , k} ~=~ \frac{k}{n} \sum_{m = k - 1}^{n - 1} F_{m , k - 1} ~~~~~~~~~~ (\forall n , k \in \N^*, \text{ with } n \geq k) ,
$$
which is easily derived by induction from \eqref{eq17}.
\end{rmq}

\noindent Next, to generate the positive integers $q_{n , k}$ ($n , k \in \N$, $n \geq k$), we can use the recurrent relation given by the following proposition:
\begin{prop}\label{p2}
For every $n , k \in \N^*$, with $n \geq k$, we have
$$
q_{n , k} ~=~ \lcm\big\{(n - m + 1) q_{m - 1 , k - 1} ~;~ k \leq m \leq n\big\} .
$$
\end{prop}
\begin{proof}
Let $n , k \in \N^*$ be fixed such that $n \geq k$. We have by definition:
$$
q_{n , k} ~:=~ \bigvee_{\begin{subarray}{c}
i_1 , \dots , i_k \in \N^* \\
i_1 + \dots + i_k \leq n
\end{subarray}} (i_1 \cdots i_k) ~=~ \bigvee_{\begin{subarray}{c}
i_1 , \dots , i_{k - 1} , i \in \N^* \\
i_1 + \dots + i_{k-1} + i \leq n
\end{subarray}} (i_1 \cdots i_{k - 1} i) . 
$$
Since for any $i_1 , \dots , i_{k - 1} , i \in \N^*$, the inequality $i_1 + \dots + i_{k - 1} + i \leq n$ implies \linebreak $i \leq n - (i_1 + \dots + i_{k - 1}) \leq n - k + 1$, then we derive that:
\pagebreak
\begin{eqnarray*}
q_{n , k} & = & \bigvee_{i = 1}^{n - k + 1} \left(\bigvee_{\begin{subarray}{c}
i_1 , \dots , i_{k - 1} \in \N^* \\
i_1 + \dots + i_{k - 1} \leq n - i
\end{subarray}} (i_1 \cdots i_{k - 1} i)\right) \\
& = & \bigvee_{i = 1}^{n - k +  1} \left(i \left(\bigvee_{\begin{subarray}{c}
i_1 , \dots , i_{k - 1} \in \N^* \\
i_1 + \dots + i_{k - 1} \leq n - i
\end{subarray}} i_1 \cdots i_{k - 1}\right)\right) \\
& = & \bigvee_{i = 1}^{n - k + 1} i q_{n - i , k - 1} \\
& = & \bigvee_{m = k}^{n} (n - m + 1) q_{m - 1 , k - 1} ~~~~~~~~~~ \text{(by setting $m = n - i + 1$)} ,
\end{eqnarray*}
as required. The proposition is proved.
\end{proof}

\noindent Similarly as for the $c_{n , k}$'s, it is also practical to arrange the $q_{n , k}$'s ($0 \leq k \leq n$) in a triangular array in which each $q_{n , k}$ is the entry in the $n$\textsuperscript{th} row and $k$\textsuperscript{th} column. Leaning on Formula of Proposition \ref{p2} and using Maple software, we get the following triangle of the $q_{n , k}$'s up to $n = 10$:

\begin{table}[h!]
\begin{tabular}{p{12mm}p{12mm}p{12mm}p{12mm}p{12mm}p{12mm}p{12mm}p{12mm}p{12mm}p{12mm}p{12mm}}
$1$ & ~ & ~ & ~ & ~ & ~ & ~ & ~ & ~ & ~ & ~ \\
$1$ & $1$ & ~ & ~ & ~ & ~ & ~ & ~ & ~ & ~ & ~ \\
$1$ & $2$ & $1$ & ~ & ~ & ~ & ~ & ~ & ~ & ~ & ~ \\
$1$ & $6$ & $2$ & $1$ & ~ & ~ & ~ & ~ & ~ & ~ & ~ \\
$1$ & $12$ & $12$ & $2$ & $1$ & ~ & ~ & ~ & ~ & ~ & ~ \\
$1$ & $60$ & $12$ & $12$ & $2$ & $1$ & ~ & ~ & ~ & ~ & ~ \\
$1$ & $60$ & $360$ & $24$ & $12$ & $2$ & $1$ & ~ & ~ & ~ & ~ \\
$1$ & $420$ & $360$ & $360$ & $24$ & $12$ & $2$ & $1$ & ~ & ~ & ~ \\
$1$ & $840$ & $5040$ & $720$ & $720$ & $24$ & $12$ & $2$ & $1$ & ~ & ~ \\
$1$ & $2520$ & $5040$ & $15120$ & $720$ & $720$ & $24$ & $12$ & $2$ & $1$ & ~ \\
$1$ & $2520$ & $25200$ & $30240$ & $30240$ & $1440$ & $720$ & $24$ & $12$ & $2$ & $1$
\end{tabular} \\
\caption{The triangle of the $q_{n , k}$'s for $0 \leq k \leq n \leq 10$}
\end{table}

To generate finally the positive integers $\lambda_n$ ($n \in \N$), we have the choice to use one of the three formulas: $\lambda_n = \lcm\{c_{n , k} ;~ 0 \leq k \leq n\}$ (according to \eqref{eq11}); $\lambda_n = q_n$ \linebreak $:= \lcm\{q_{n , k} ;~ 0 \leq k \leq n\}$ (according to Theorem \ref{t4}) or $\lambda_n = \prod_{p \text{ prime}} p^{\lfloor\frac{n}{p}\rfloor}$ (according to Theorem \ref{t4}). The calculations give the following first terms of the sequence ${(\lambda_n)}_{n \in \N}$:
\begin{multline*}
1 ~,~ 1 ~,~ 2 ~,~ 6 ~,~ 12 ~,~ 60 ~,~ 360 ~,~ 2520 ~,~ 5040 ~,~ 15120 ~,~ 151200 ~,~ \dots
\end{multline*}

\end{document}